\documentclass[letterpaper,12pt]{amsart}

\usepackage{url}
\usepackage{graphicx}
\usepackage{graphics}
\usepackage{amsmath}
\usepackage{float}
\usepackage{tikz-cd}
\usepackage{amssymb,amsmath}
\usepackage{enumerate}
\usepackage{hyperref}

\usepackage{geometry}
\theoremstyle{plain}
\newtheorem{theorem}{Theorem}[section]

\newtheorem{lemma}[theorem]{Lemma}
\newtheorem{definition}[theorem]{Definition}
\newtheorem{proposition}[theorem]{Proposition}
\newtheorem{corollary}[theorem]{Corollary}

\newtheorem{notation}[theorem]{Notation}

\theoremstyle{remark}

\newcommand{\hxk}{H_x^{(k(n-1)+1)}}
\newcommand{\hyk}{H_y^{(k(n-1)+1)}}
\newcommand{\hxn}{H_x^{(n)}}
\newcommand{\hyn}{H_y^{(n)}}

\newcommand{\rhoxk}{\rho_x^{(k(n-1)+1)}}
\newcommand{\rhoyk}{\rho_y^{(k(n-1)+1)}}
\newcommand{\rhoxn}{\rho_x^{(n)}}
\newcommand{\rhoyn}{\rho_y^{(n)}}

\newcommand{\E}{E_{k,n}}

\newcommand{\orb}{\hbox{orb}_n}
\newcommand{\pixn}{\pi_x^{(n)}}
\newcommand{\pixnk}{\pi_x^{(k(n-1)+1)}}

\title{Representations of Higman-Thompson groups from Cuntz algebras}

\author{Francisco Ara\'{u}jo}
\address{Department of Mathematics,
Instituto Superior T\'{e}cnico, University of Lisbon,
Av.\ Rovisco Pais 1, 1049-001 Lisboa, Portugal.}
\email{francisco.araujo@tecnico.ulisboa.pt}

\author{Paulo R.\ Pinto}
\address{Department of Mathematics, CAMGSD, Instituto Superior T\'{e}cnico, Univ.\ Lisboa,
Av.\ Rovisco Pais 1, 1049-001 Lisbon, Portugal.}
\email{ppinto@math.tecnico.ulisboa.pt}

\begin{document}
	
	\maketitle

	\begin{abstract}
		Every representation of the Cuntz algebra $\mathcal{O}_n$ leads to a unitary representation of the Higman-Thompson group $V_n$. We consider the family $\{\pi_x\}_{x\in [0,1[}$ of permutative representations of $\mathcal{O}_n$ that arise from the interval map $f(x)=nx$ (mod 1) acting on the Hilbert space that underlies each orbit, and
then study the unitary equivalence and the irreducibility of the corresponding family $\{\rho_x\}_{x\in [0,1[}$ of representations of Higman-Thompson group $V_n$, showing that that these representations are indeed irreducible and moreover $\rho_x$ and $\rho_y$ are equivalent if and only if the orbits of $x$ and $y$ coincide.
	\end{abstract}

\medskip

{\it Keywords}:\ Higman-Thompson groups, Cuntz algebras, Representations

\medskip

MSC\ {46L05, 20F65, 22D25, 20F38, 37A55}

	\tableofcontents

	\section{Introduction}

In this paper we investigate the interplay between representations of Cuntz algebras (see \cite{cuntz}) and representations of Higman-Thompson groups $V_n$ (see \cite{cfp,Higman,brown}). In the context of the interval map $f(x)=nx$ (mod 1), we explain how every point $x\in [0,1[$ gives rise to a representation $\rho_x$ of the Higman-Thompson group $V_n$ acting on the Hilbert space associated to the orbit of $x$ (which shows up via a representation of the Cuntz algebra $\mathcal{O}_n$ on the same Hilbert space). Then the main aim is to study
the (unitary) equivalence and the irreducibility of this family ${\{\rho_x\}}_{x\in [0,1[}$ of representations of Higman-Thompson groups, leading to the main result:
$\rho_x$ and $\rho_y$ are equivalent if and only if the orbits (under $f$) of $x$ and $y$ coincide. In order to prove this, we need to use the knowledge of the underlying representations of the Cuntz algebras together with both embeddings of Cuntz algebras and embeddings of Higman-Thompson groups which leads us to extend the results obtained in \cite{ppinto} (for the particular $n=2$ case) to the general case $n$. The proofs for the general case typically required a different methods of the ones for the $n=2$ case \cite{ppinto}.

Our work is motivated by the recent renewed interest in the well developed theory of representations of the Cuntz algebras $\mathcal{O}_n$ and its application to the representation theory of the Higman-Thompson groups. New examples of concrete representations of these groups are regarded as being important not only to tackle some famous open problems (e.g.\ the (non)amenability of $F_2$) but also in pursuing the recent Jones' machinery where subfactors, Higman-Thompson groups and conformal field theory become linked together, see e.g.\ \cite{Vjones,brohier2021}.

We remark that our representations are related to Pythagorean representations in the sense of \cite{jones2019}. Namely, the $n$ generators $t_1,...,t_n$ of the C*-algebra $P_n$ (put forward in \cite{jones2019} as part of Jones' technology to produce representations of Higman-Thompson groups from those of $P_n$, the so-called Pythagorean representations) decomposes the space in $n$ subspaces
$$t_1^\ast t_1+t_2^\ast t_2 ...+ t_n^\ast t_n=\mathbf{1},$$
 whereas the $n$ generators $s_1,..., s_n$ of the Cuntz algebra $\mathcal{O}_n$ provides an orthogonal decompositions of the space
 $$s_1^\ast s_1+...+s_ns_n^\ast=\mathbf{1},\ \ s_i^\ast s_i=\delta_{ij}\mathbf{1}.$$
  Then, $P_n$ has many quotients, with $\mathcal{O}_n$ being one of them, mapping $t_i\to s_i^\ast$. So every representation of $\mathcal{O}_n$ gives rise to a representation of $P_n$, and thus of the Higman-Thompson groups.



In the 1960's, R.J.\ Thompson introduced the so called Thompson groups $F_2 \subset T_2\subset V_2$, see \cite{cfp}. Then G. Higman generalized $V_2$ to a family $\{V_n\}_{n\in\mathbb{N}}$ of finitely presented discrete groups \cite{Higman}. K.S.\ Brown extended the groups $V_n$ to triplets $F_n\subset T_n\subset V_n$, see \cite{brown}. This groups are now called Higman-Thompson groups. These are countable and discrete groups and fairly easy to define as certain piecewise linear maps from the interval $[0,1]$ onto itself, as we shall review below. Almost every question related to these groups is a challenge, typically harder for the smaller groups, for example it is still an open problem whether $F_n$ is an amenable group or not whereas the others contain copies of free groups thus they are nonamenable, see e.g.\ \cite{olesen}. We note that several approximation properties for groups are based on asymptotic behaviour of matrix coefficients of representations of the groups in question. Besides this, V.F.R.\ Jones has recently developed a machinery to produce certain unitary representations of Higman-Thompson groups leading him to yield (unoriented and oriented) knot and link invariants \cite{Vjones} by playing with the matrix coefficients of his representations. It is therefore of interest to determine as much information as possible about the representations theory of the Higman-Thompson groups (the amenability problem springs to mind).

On the other hand we have a fairly well developed theory of Cuntz algebra $\mathcal{O}_n$ representations with an enormous boost since the seminal work by Bratteli and Jorgensen \cite{Jorgensen}. Recall that a representations of $\mathcal{O}_n$ on a complex Hilbert space is a family of isometries $S_1,...,S_n$ acting on $H$ with orthogonal ranges such that $H$ is subdivided into these ranges. The permutative representations of $\mathcal{O}_n$, where the isometries permute the vectors of a fixed Hilbert basis, have a host of applications, for example to fractals, wavelets, dynamical systems \cite{Jorgensen,KawamuraIn,ar2008}.

The link between these two subject, representations of Cuntz algebras and Higman-Thompson groups, was unveiled in \cite{Nek,birget0}, where a realization of the Higman-Thompson group $V_n$ as a subgroup of the unitary group  of $\mathcal{O}_n$ was fabricated. It is naturally expected that tools from the representation theory of Cuntz algebras can be used in the study of the representation theory of the Higman-Thompson groups. This interplay is what we investigate in this paper.
Indeed we clarify why every representation $\pi$ of the Cuntz algebra $\mathcal{O}_n$ on a Hilbert space $H$ gives rise to a unitary representation $\rho_\pi$ of the Higman-Thompson group $V_n$ on the same Hilbert space:
\begin{equation}
  \pi\in\hbox{Rep}(\mathcal{O}_n, H)\ \ \longmapsto \ \ \rho_\pi\in\hbox{Rep}(V_n, H),
\label{eqmap}
\end{equation}
see Theorem \ref{thm:rhorep}.
This is accomplished by regarding the elements of $V_n$ either as piecewise linear maps or as tables.
Namely, the Higman-Thompson group $V_n$ is the set of piecewise linear functions $g: [0,1)\to [0,1)$ whose slopes are powers of $n$ and the non-differentiable points live in the $n$-adic set $\mathbb{Z}[{\frac{1}{n}}]$ (with composition being the group operation). Every such function can be represented by a table $\begin{bmatrix}
	a_1 & a_2 & \ldots & a_m \\
	b_1 & b_2 & \ldots & b_m
	\end{bmatrix}$ where $a_i$ and $b_i$ are multi-indices in $\{1,...,n\}$ related with the $x$- and $y$-axis partitions of $f$, as we will review below. Then
$$\begin{bmatrix}
	a_1 & a_2 & \ldots & a_m \\
	b_1 & b_2 & \ldots & b_m
	\end{bmatrix}\ \ \longrightarrow\ \ s_{a_1}a_{a_1}^\ast +...+s_{a_m}a_{a_m}^\ast $$
gives an embedding of $V_n$ into the (unitaries) C*-algebra $\mathcal{O}_n$, where for example $s_{132}=s_{1}s_3s_2$. Hence every representation $\pi: \mathcal{O}\to B(H)$ on a Hilbert space $H$ gives rise to a representation of $V_n$ by taking the restriction of $\pi$ into the image of the above embedding of $V_n$. Therefore the representations $\rho_\pi: V_n\to B(H)$ can be seen as unitary operators on the Hilbert space $H$ as follows
$$\begin{bmatrix}
	a_1 & a_2 & \ldots & a_m \\
	b_1 & b_2 & \ldots & b_m
	\end{bmatrix}\ \ \longrightarrow\ \ \pi(s_{a_1}a_{a_1}^\ast +...+s_{a_m}a_{a_m}^\ast).$$
Then several questions can be asked, for example, if the unitary equivalence or irreducibility is preserved in \eqref{eqmap}. We study this problem by considering the permutative representations $\pi_x$ of $\mathcal{O}_n$ on the Hilbert space $H_x$ that encode the orbit of point $x\in [0,1[$ with respect to the interval map $f(x)=nx$ (mod 1) as in \cite{ar2008}, leading to the representation $\rho_x:=\rho_{\pi_x}$ of $V_n$. It is usually difficult to distinguish equivalent classes and decide the irreducibility of the Higman-Thompson groups. We however succeeded in characterizing the inequivalent classes and irreducibility of this family $\{\rho_x\}_{x\in [0,1[}$ of representations.

The plan of the rest of the paper is as follows. In Sect.\ \ref{preliminaries} we review some background
on the representation theory of operator algebras and of group theory together with the definition of the Higman-Thompson groups as piecewise linear maps and as tables.

In Sect.\ \ref{Embeddings of the Higman-Thompson's groups} we first provide and embedding $\iota: \mathcal{O}_{k(n-1)+1}\to \mathcal{O}_n$ as in Lemma \ref{thm:OninO2gen}. This embedding together with the embedding $\Psi: V_n\to \mathcal{O}_n$ in \cite{Nek} leads us to an embedding $\E: V_{k(n-1)+1} \to V_n$, which was also proved by Higman \cite{Higman}. This is proved in Theorem \ref{thm:EmbFinT} where we also show that $\E$ maps $T_{k(n-1)+1}$ into $T_n$ and $F_{k(n-1)+1}$ into $F_n$. In Subsect.\ \ref{HTtables} we clarify the Higman-Thompson group elements $g\in V_n$ as a piecewise map $g: [0,1[\to[0,1[$ and as a table $g$=$\begin{bmatrix}
	a_1 & a_2 & \ldots & a_m \\
	b_1 & b_2 & \ldots & b_m
	\end{bmatrix}$, in particular we write explicitly the piecewise map $g\in V_n$
from a table, see Eq.\ \eqref{eqtablePL}.

In Sect.\ \ref{Representations of $V_n$ in $H$} we derive the main results of the paper, starting with the construction of the representation $\rho_\pi$ of the Higman-Thompson group $V_n$ from a Cuntz algebra $\mathcal{O}_n$ representation $\pi$. We remark that since $\pi: \mathcal{O}_n\to B(H)$ is a *-homomorphism, $\pi$ is automatically continuous w.r.t.\ the norm topologies, and since $V_n$ is a discrete group $\rho_\pi$ is automatically continuous.  Then from Subsect.\ \ref{sect:rhoxorb} we concentrate on the family of representations $\{\rho_x\}_{x\in [0,1[}$ of $V_n$ (with $\rho_x:=\rho_{\pi_x}$ that arises from the 1-dimensional dynamical systems that underlies the interval maps $f(x)=nx$ (mod 1), with $n\in\mathbb{N}$. Here the Hilbert space is attached to the (generalized) orbit orb$(x)=\bigcup_{k\in\mathbb{Z}} f^k(x)$. In Theorem \ref{thm:bigs4} we show that the action $g.y:=g(y)$ is well defined for $g\in V_n$ and $y\in\hbox{orb}(x)$ and moreover the underlying representation of $V_n$ coincides with that of $\rho_x$ -- this is proved in several steps. Then we can use the embedding $\E: V_{k(n-1)+1}\to V_n$ and obtain a representation $\rhoxn\circ \E$ of  $V_{k(n-1)+1}$ where we denote by $\rhoxn$ the already introduced representation $\rho_x$ of $V_n$ on $\hxn:=H_x$. The unitary operator $U: \hxk\to U(\hxn)\subset \hxn$ introduced in Definition \ref{def:U} is used to prove that in fact $\rhoxn\circ \E$ and $\rhoxk$ are unitarily equivalent as in Theorem \ref{thm:Upi=pi'U}. Using a natural embedding $\iota^\prime: \pixnk(\mathcal{O}_{k(n-1)+1})\rightarrow B(\hxn)$ as in Definition \ref{def:iota'}, then using \cite{ar2008} we show in Corollary \ref{crl:pixy<->xy} that the representations $\iota^\prime\circ \pi_x$ and $\iota^\prime\circ \pi_y$ of $\mathcal{O}_{k(n-1)+1}$ are unitarily equivalent if and only if $\text{orb}(x) = \text{orb}(y)$. We then study in Subsect.\ \ref{sec: unitequiv} the more involved unitary equivalence and irreducibility of the Higman-Thompson groups representations ${\{\rho_x\}}_{x\in [0,1[}$, so that Theorem \ref{thm:Crho=piO} shows that the C*-algebra $C^*_{\rhoxn}(V_n)$ generated by ${\rhoxn}(V_n)$ equals the C*-algebra  $\pi_x(\mathcal{O}_n)$ on $H_x$, i.e.\
\[ C^*_{\rhoxn}(V_n) = \pi_x(\mathcal{O}_n) \]
where the non-amenability of the underlying actions plays a role as well as the recent embedding result of $V_2$ into $V_n$ as in \cite{birget}.
This leads to Theorem \ref{thm:rho<->xy} where we prove that $\rho_x$ and $\rho_y$ are equivalent representations of $V_n$ if and only if $\text{orb}(x) = \text{orb}(y)$, i.e.\
$$\rho_x \sim \rho_y\ \ \hbox{if and only if}\ \ \hbox{orb}(x)= \hbox{orb}(y),$$
and moreover they are irreducible.

\section{Preliminaries}
	\label{preliminaries}
	
	We start by defining the Cuntz algebras, a set of simple universal C*-algebras first introduced by Cuntz in \cite{cuntz}. The Cuntz algebra $\mathcal{O}_n$ is defined as the (universal) C*-algebra generated by the $n$ isometries $\{s_1,s_2,\ldots,s_n\}$ satisfying:
	\begin{equation}
		\label{def:CuntzAlg}
		\sum_{j=1}^{n}s_js_j^* = 1 \hspace*{30pt} s_i^*s_j = \delta_{ij}\ 1
		\ \hbox{for any i,j}\  \in \{1,2,...,n\}
	\end{equation}
where $1$ denotes de identity.
	
	Let $A$ be a C*-algebra. Given a Hilbert space $H$, we denote by $B(H)$ the C*-algebra of linear bounded operators in $H$, where the product of $B(H)$ is the composition of operators and $1$ also denotes the identity of $B(H)$. A representation of $A$ is a *-homomorphism $\pi:A\rightarrow B(H)$ and thus satisfies $\pi(a^\ast)=\pi(a)^\ast$ for $a\in A$. $\pi$ is also automatically continuous for the norm topologies.
The representation $\pi$ is said to be irreducible if it has no invariant subspaces, that is, if there is no non trivial subspace $K \subset H$ such that $\pi(K) \subset K$.
This is equivalent to the commutant of $\pi$ being $\mathbb{C}1$, that is, $k\pi(a) = \pi(a)k$ (for every $a \in A$) if and only if $k \in \mathbb{C}1$ (see for example \cite{pedersen})

	Given two Hilbert spaces $H_1$, $H_2$, we say that $U: H_1 \rightarrow H_2$ is unitary, if $UU^* = 1$ and $U^*U = 1$, where $U^*$ denotes the adjoint operator of $U$. This is equivalent to $U$ being onto, and satisfying $\langle Ux,Uy \rangle_{H_2} = \langle x,y \rangle_{H_1}$ for any $x,y \in H_1$. Given a C*-algebra $A$, the representations $\pi_1:A\rightarrow B(H_1)$ and $\pi_2:A\rightarrow B(H_2)$ are said to be unitarily equivalent if there exists a unitary $U: H_1\to H_2$ such that
\begin{equation}
U\pi_1(a) = \pi_2(a)U\  \hbox{for any}\ a \in A.
\end{equation}
	
	Given a group $G$ and a Hilbert space $H$, there are analogous properties for a representation $\rho$ of $G$, that is, a group homomorphism from $G$ to $B(H)$. We say that $\rho$ is irreducible if there is no non trivial subspace $K \subset H$ such that $(\pi(g))(K) \subset K$ for all $g \in G$. This is equivalent to $\rho'(G) = \mathbb{C}1$. We say that $\rho$ is unitary, if for any $g \in G$, $\rho(g^{-1}) = \rho(g)^*$. Finally, given two Hilbert spaces $H_1$ and $H_2$ not necessarily different, we say that the representations $\rho_{1}:G\rightarrow B(H_1)$, $\rho_{2}:G\rightarrow B(H_2)$ are unitarily equivalent, and write $\rho_{1} \sim \rho_{2}$, if there is a unitary operator $U: H_1 \rightarrow H_2$ such that $U\rho_1(g) = \rho_{2}(g)U$ for every $g \in G$. Given a representation $\rho$ of a discrete group $G$ in a Hilbert space $H$, let $C^*_{\rho}(G)$ denote the C*-subalgebra of $B(H)$ generated by $\rho(G)$ , so that
\begin{equation}
C^*_{\rho}(G) = \overline{\text{span}(\{\rho(g):g \in G\})}^{||\cdot ||_{B(H)}}.
\end{equation}
	
We now define the Higman-Thompson groups, see \cite{cfp,Higman}. There are several equivalent realizations of this groups. In this paper, we will consider them as linear piecewise maps in the interval $[0,1[$, and as tables. The relation between these is made explicit in section \ref{Representations of $V_n$ in $H$}. Fix a certain $n \geq 2$, and let $ M = \{\frac{a}{n^k}:  a,k \in \mathbb{N} \mbox{ and } 0 \leq a < n^k\}$. We now define the Higman-Thompson groups, as a particular case of the groups defined in \cite{pie}.
	
	\begin{definition}
		The Higman-Thompson group $V_n$ is the group of piecewise linear maps $g: [0,1[ \rightarrow [0,1[$ such that:
		\begin{enumerate}
			\item  $g$ is bijective in $[0,1[$.
			\item $g(M)=M$.
			\item $g'(x) = n^k$ for some $k$ in the points where it is differentiable.
			\item If $g$ is not differentiable in $x$, then $x \in M$.
		\end{enumerate}
		The Higman-Thompson group $T_n$ is the subset of $V_n$ such that $g$ has at most one discontinuity. The Higman-Thompson group $F_n$ is the subset of $T_n$ such that $g$ has no discontinuities.
\label{def:VTF}
\end{definition}
	
	We will now describe the realization of the Higman-Thompson groups as tables, which can also be found in e.g.\ \cite{Nek}. An alphabet $A$ is a finite set and its elements are called letters. We will denote by $A^*$ the free monoid generated by $A$, that is, the set of finite sequences $a_1a_2...a_m$ with $a_i \in A$.  The length of a  word  $w\in A^*$ is the number of letters occurring in $w$. For example, $w=a_1a_2a_2$ has length $3$.  Given $a\in A^*$, we will denote by $aX$ the set $\{ax: x \in X\}$. The set of words on $A$ of infinite length will be denoted by $A^{\omega}$.  Given $x \in A^*$, $y \in A^\omega\cup A^*$, we say that $x$ is a prefix of $y$, if there is a $z\in A^\omega\cup A^*$ such that $y = xz$.
	
	Let $A$ be an alphabet with $n$ letters. An admissible language $L:=\{a_1,\ldots,a_m\}$ is a subset of $A^*$ such that $A^\omega=\bigcup_{i=1}^{m} a_iA^{\omega}$ and no $a_i$ is a prefix for $a_j$, for all $i,j\in\{1,\ldots,m\}$. Symbolically, $a_iA^\omega \cap a_jA^\omega=\emptyset$, for  pairwise different $a_i$, $a_j$.
	
	We will be concerned with admissible transformations,
\begin{equation}	
g = \begin{bmatrix}
	a_1 & a_2 & \ldots & a_m \\
	b_1 & b_2 & \ldots & b_m
	\end{bmatrix},
\label{eq:tableV}
\end{equation}
where $\{a_1,\ldots,a_m\}$ and $\{b_1,\ldots,b_m\}$ are admissible languages.
	
	Any admissible transformation induces a permutation of  $A^\omega$ as follows: for all $w\in A^\omega$,  $T(w)=T(a_iu)=T(a_i)u$. This is well defined because for every $w\in A^\omega$ there is one and only one $a_i\in \{a_1,\ldots, a_m\}$ and $u\in A^\omega$ such that $w=a_iu$.
	
		The composition of two admissible transformations is admissible.
		Furthermore, the identity transformation on an admissible set induces the identity permutation on $A^\omega$, and switching the rows in $T$ yields the inverse permutation.  It is worth observing  that for any admissible transformation there is an infinite number of tables one can associate to it.
	
	The Higman-Thompson group $V_n$, is the group of all admissible transformations  of a $n$-letters alphabet. We may assume, without loss of generality, that $b_1 < b_2  < \ldots < b_m$, where $\le$ is the lexicographic order. In order for a table to be associated to a map that belongs to $T_n$, we additionally need to have: $a_i < a_{i+1}  < \ldots < a_m < a_1 < \ldots < a_{i-1}$ for some $i \in \{1,2,\ldots,m\}$. For the table to be associated to a map that belongs to $F_n$, we need to have $i = 1$.

\section{Embeddings of the Higman-Thompson's groups}
\label{Embeddings of the Higman-Thompson's groups}
	
	In this section, we use some known embeddings of the Cuntz Algebras to prove some embedding results for the Higman-Thompson groups thus retrieving a result by Higman (see \cite{Higman}). These embeddings will be critical in Section \ref{Representations of $V_n$ and Embeddings}.
	
We now setup an embedding of the Cuntz algebras which can be traced back in \cite{cuntz}.
For completeness we provide its proof.
	
	\begin{lemma}
		\label{thm:OninO2gen}
		Given any integer $k \geq 1$, the Cuntz Algebra $\mathcal{O}_{k(n-1)+1}$ is embedded in $\mathcal{O}_{n}$. An embedding is the map $\iota: \mathcal{O}_{k(n-1)+1}\rightarrow \mathcal{O}_{n}$ satisfying:		
		\[ \iota(\hat{s_1}) = s_1^k \hspace*{30pt} \iota(\hat{s}_{1+i(n-1)+(j-1)}) = \iota(\hat{s}_{i(n-1)+j}) = s_1^{k-i}s_j \]
		for $0\leq i < k$, $2 \leq j \leq n$.
	\end{lemma}
	
	\begin{proof}
		We must prove that the set $\{\iota(\hat{s_i}): 1\leq i \leq k(n-1)+1\}$ satisfy the Cuntz relations. We start by proving that:
		\[ s_1^k(s_1^*)^k + \sum_{i=1}^{k}\sum_{j=2}^{n} s_1^{k-i}s_js_j^*(s_1^*)^{k-i} = 1 \]
		
		We prove the result by induction in $k$. For $k = 1$ the result comes from the definition of $\mathcal{O}_n$. Now, suppose the result is true for $k$. Then: 	
		\[\begin{array}{lll}
		& & \displaystyle s_1^{k+1}(s_1^*)^{k+1} + \sum_{i=1}^{k+1}\sum_{j=2}^{n} s_1^{k+1-i}s_js_j^*(s_1^*)^{k+1-i}\\
		     & = & \displaystyle s_1s_1^{k}(s_1^*)^{k}s_1^* + \sum_{i=1}^{k}\sum_{j=2}^{n} s_1s_1^{k-i}s_js_j^*(s_1^*)^{k-i}s_1^* + \sum_{j=2}^{n}s_js_j^* \\
		     & = & \displaystyle s_1\left(s_1^{k}(s_1^*)^{k} + \sum_{i=1}^{k}\sum_{j=2}^{n} s_1^{k-i}s_js_j^*(s_1^*)^{k-i}\right)s_1^* + \sum_{j=2}^{n}s_js_j^* \\
		     & = & \displaystyle s_1s_1^* + \sum_{j=2}^{n}s_js_j^* = 1.
		\end{array}
		  \]
		One can easily verify that $\iota(\hat{s_i})\iota(\hat{s_j}) = \delta_{ij}$
	\end{proof}
	
	We are now going to use the Cuntz algebra embedding $\iota$ of Lemma \ref{thm:OninO2gen} to yield an embedding of the Higman-Thompson groups. To do this, we will use the map $\Psi_n :V_n \rightarrow \mathcal{O}_n$ defined as follows
\begin{equation}	
 \Psi_n(g) = \Psi\left(
	\begin{bmatrix}
	a_1 & a_2 & \ldots & a_m \\
	b_1 & b_2 & \ldots & b_m
	\end{bmatrix}\right)
	 = s_{a_1}s^*_{b_1} + s_{a_2}s^*_{b_2} + \ldots + s_{a_m}s^*_{b_m}.
\label{eq:Psi}
\end{equation}
as defined in \cite{Nek}, where $g\in V_n$ is represented by a table as explained in \eqref{eq:tableV}.
    In \cite{Nek}, it is proven that $\Psi$ is a faithful unitary representation of $V_n$ in $\mathcal{O}_n$. We now check that $\iota$ maps $\Psi_{k(n-1)+1}(V_{k(n-1)+1})$ to $\Psi_n(V_n)$.

Using this, we explain how to define a group embedding $\E:V_{k(n-1) +1}\rightarrow V_n$  such that the following diagram commutes
    \[ \begin{tikzcd}
    \Psi_{k(n-1)+1}(V_{k(n-1)+1} \arrow{r}{\iota})  & \Psi_n(V_n) \arrow{d}{\Psi_n^{-1}} \\%
    V_{k(n-1)+1}\arrow[swap]{u}{\Psi_{k(n-1)+1}} \arrow[swap]{r}{\E}& V_n
    \end{tikzcd}
    \]
and then study the corresponding Higman-Thompson group embeddings.

    To do this, we need to define some auxiliary maps. In this section, we will denote by $X$ and $Y$ the alphabets $X = \{1,2,...,n\}$ and $Y = \{1,2,...,k(n-1)+1\}$. In these, we then define a lexicographic order as it should be expected. Given two letters, $i$ and $j$, we write $i<j$ if the natural number i is less than the natural number j. Then, given $v,w \in X^*$, we write $v<w$ if there is a $j$ such that $v_j < w_j$, and $v_i = w_i$ for all $i<j$. With this, we define our auxiliary functions.
	
	\begin{definition}
		\label{def:gamma,f}
		Let $\gamma:Y \rightarrow X^*$ be the map such that, for all $y \in Y$:
		\[ \iota(\hat{s}_y) = s_{\gamma(y)} \]
		Also,  let $f:Y^*\rightarrow X^*$ be such that $$f(u) = f(u_1u_2\ldots u_m) = \gamma(u_1)\gamma(u_2)\ldots\gamma(u_m).$$
		That is, the map such that:
		\[ \iota(\hat{s}_u) = \iota(\hat{s}_{u_1u_2\ldots u_m}) = \iota(\hat{s}_{u_1})\ldots\iota(\hat{s}_{u_m}) = s_{\gamma(u_1)}\ldots s_{\gamma(u_m)} = s_{\gamma(u_1)\gamma(u_2)\ldots\gamma(u_m)} = s_{f(u)} \]
	\end{definition}

	So, for example, since $\iota(\hat{s_1}) = \underbrace{s_1\ldots s_1}_{k \text{ times}}$, we have $\gamma(1) = \underbrace{(1)\ldots(1)}_{k \text{ times}}$. We are now in a position to describe our embedding.
	
We define $\E = \Psi_n \circ \iota \circ \Psi_{k(n-1)+1}$ such that
\begin{equation}
\E:\begin{bmatrix}
		u_1 & u_2 & \ldots & u_m \\
		v_1 & v_2 & \ldots & v_m
		\end{bmatrix} \mapsto
		\begin{bmatrix}
		f(u_1) & f(u_2) & \ldots & f(u_m) \\
		f(v_1) & f(v_2) & \ldots & f(v_m)
		\end{bmatrix}.
\label{eq:E}
\end{equation}
	
In order to show that $\iota$ maps $\Psi_{k(n-1)+1}(V_{k(n-1)+1})$ to $\Psi_n(V_n)$, and thus that $\E$ is well defined,  we need to prove that $f$ maps admissible languages in $Y^*$ to admissible languages in $X^*$, that is
	\[Y^{\omega} = \bigcup_{i=1}^{m} u_iY^{\omega} = \bigcup_{i=1}^{m} v_iY^{\omega} \Rightarrow X^{\omega} = \bigcup_{i=1}^{m} f(u_i)X^{\omega} = \bigcup_{i=1}^{m} f(v_i)X^{\omega}\]
	and $u_iY^{\omega}\cap u_jY^{\omega} = \emptyset$ implies that $f(u_i)Y^{\omega}\cap f(u_j)Y^{\omega} = \emptyset $. Furthermore, in order to prove that $\E$  is also an embedding of  $T_{1+k(n-1)}$ and $F_{1+k(n-1)}$ in $T_{n}$ and $F_{n}$ respectively, one must show that $f$ preserves the lexicographic order, that is, $ a < b \Rightarrow f(a) < f(b)$, since this implies that:
	\begin{eqnarray*}
	  u_i < u_{i+1}  < \ldots < u_m < u_1 < \ldots < u_{i-1} \Rightarrow f(u_i) < f(u_{i+1})  < \ldots\\
	< f(u_m) < f(u_1) < \ldots < f(u_{i-1})\ \ \hbox{for all} \ i.
	\end{eqnarray*}
	
	We prove an auxiliary result.
	
	\begin{lemma}		
\label{lma:f(a)=f(b)->a=b}
${}_{}$\par
\begin{enumerate}
	\item	$f$ is injective.
\item If $f(a)$ is a prefix of $f(b)$, then $a$ is a prefix of $b$.
\end{enumerate}
	\end{lemma}
	
	\begin{proof}
		(1) We will prove that if $a \neq b$, then $f(a) \neq f(b)$. Start by noticing that for all $i,j \in Y$, $\gamma(i)$ is not a prefix of $\gamma(j)$ if $i \neq j$ and that $\gamma$ is injective. Suppose $a \neq b$. Then let $j$ be the first index such that $a_j \neq b_j$. It follows that $\gamma(a_j) \neq \gamma (b_j)$. We have that $f(a) = \gamma(a_1)\gamma(a_2)\ldots\gamma(a_j)\ldots$ and $f(b) = \gamma(b_1)\gamma(b_2)\ldots\gamma(b_j)\ldots$. If $f(a) = f(b)$, then we would need to have that $\gamma(a_j)$ is a prefix of $\gamma(b_j)$ or vice versa, which is impossible.

(2) We prove the result by induction on the length of $b$. If $b \in Y$, $f(a) = \gamma(a)$ and $f(b) = \gamma(b)$, so if $f(a)$ is a prefix of $f(b)$, $\gamma(a) = \gamma(b)$ and thus $a = b$, from which it follows that $a$ is a prefix of $b$. Now, suppose that $b$ has length $n$. We either have that $f(a) = f(b)$, or that $f(a)$ is a prefix of $f(b_1\ldots b_{n-1})$. If $f(a) = f(b)$, then by  part (1), we have that $a = b$. If $f(a)$ is a prefix of $f(b_1\ldots b_{n-1})$, then by the induction hypothesis we have that $a$ is a prefix of $b_1\ldots b_{n-1}$, and hence a prefix of $b$.
	\end{proof}

We have now the tools to prove the requested Higman-Thompson group embeddings.
	
\begin{theorem}
\label{thm:EmbFinT}
The map $\E$ is an embedding of $V_{1+k(n-1)}$, $T_{1+k(n-1)}$ and $F_{1+k(n-1)}$ in
$V_n$, $T_{n}$ and $F_{n}$ respectively.
\end{theorem}
	
	\begin{proof}
		We start by showing that $\E$ is well defined by proving that $\iota$ (see Lemma \ref{thm:OninO2gen}) maps $\Psi_{k(n-1)+1}(V_{k(n-1)+1})$ to $\Psi_n(V_n)$.
	
	 Let $r$ denote the length of the biggest $u_i$. Let $y \in Y^*$ be any word of length more than $r$. To $y$ add an infinite number of $1$s. This new word  must have as prefix one of the $u_j$. As $y$ is larger than any $u_i$ it follows that $u_j$ is a prefix of $y$.  Now, notice that
		\[ X^{\omega} = \bigcup_{i=1}^{1+k(n-1)}f(i)X^{\omega}. \]
		
		Let $x \in X^{\omega}$. By the result above, we have that $x = f(w_1)t_1$, $t_1 \in X{^\omega}$, $w_1\in Y$. Repeating the steps $r$ times we conclude $x = f(w_1)\ldots f(w_{r+1})t_{r+1} = f(w_1 \ldots w_{r+1})t_{r+1} = f(w)t_{r+1}$. Since $w \in Y^*$ is a word of size $r+1$, we conclude it must have a prefix $u_i$. Thus, $f(w)$ must have a prefix $f(u_i)$ and thus $x \in f(u_i)X^{\omega}$. It follows that		
		\[ X^{\omega} = \bigcup_{i=1}^{m} f(u_i)X^{\omega}.\]
		Now we claim that if $i \neq j$, then $f(u_i)X^{\omega} \cap f(u_j)X^{\omega} = \emptyset$. We have that $u_iX^{\omega} \cap u_jX^{\omega} = \emptyset$ and that $aX^{\omega} \cap bX^{\omega} \neq \emptyset$ is the same as saying that $a$ is a prefix of $b$ or vice versa. Suppose $f(u_i)X^{\omega} \cap f(u_j)X^{\omega} \neq \emptyset$. Then, we must have that $f(u_i)$ is a prefix of $f(u_j)$, or vice versa. Suppose without loss of generality, that $f(u_i)$ is a prefix of $f(u_j)$. By Lemma \ref{lma:f(a)=f(b)->a=b}, we conclude that $u_i$ is a prefix of $u_j$. But this implies that $u_iX^{\omega} \cap u_jX^{\omega} \neq \emptyset$, a contradiction.	
		
	   Given that every Cuntz algebra is simple, $\iota$ must be an injective map. Given that $\Psi_m$ is a faithful representation, it is also an injective map for any $m$. Considering $\iota$ as a group homomorphism, we conclude that $\E = \Psi_n \circ \iota \circ \Psi_{k(n-1)+1}$ is an injective homomorphism from $V_{k(n-1)+1}$ to $V_n$.
       It remains to show that $f$ preserves the lexicographic order, and therefore $\E$ is also an embedding for $T_n$ and $F_n$.  First we show that $f$ preserves the lexicographic order of the letters, and then that of words. Suppose $a$ and $b$ are letters in $Y$. Then, $f(a) = \gamma(a)$. We have two cases. If $a = i(n-1)+ \alpha$, $b = i(n-1)+ \beta$, then
		\[ a<b \Rightarrow \alpha < \beta \Rightarrow \underbrace{(1)\ldots(1)}_{(k-i) \text{ times}}(\alpha) < \underbrace{(1)\ldots(1)}_{(k-i) \text{ times}}(\beta) \Rightarrow f(a) < f(b). \]
		Else, if $a = i(n-1)+ \alpha$, $b = j(n-1)+ \beta$, then
		\[ a<b \Rightarrow i < j \Rightarrow \underbrace{(1)\ldots(1)}_{(k-i) \text{ times}}(\alpha) < \underbrace{(1)\ldots(1)}_{(k-j) \text{ times}}(\beta) \Rightarrow f(a) < f(b). \]
		Let $w$ and $z$ be words with letters in $Y$. We can write $w = w_1\ldots w_m$ and $z = z_1\ldots z_l$. Then, by definition, $w<z$ means that there is a $j$ such that $w_i = z_i$ for any $i<j$, and that $w_j < z_j$. Thus, $\gamma(w_i) = \gamma(z_i)$ for all $i<j$, and $\gamma(w_j) < \gamma(z_j)$. Therefore, $\gamma(w_1)\ldots\gamma(w_{j-1})\gamma(w_j)\ldots < \gamma(z_1)\ldots\gamma(z_{j-1})\gamma(z_j)\ldots$ which implies that $f(u) < f(v)$.
	\end{proof}
	
	For other proof of the embedding $V_{k(n-1)+1}$ into $V_{n}$, see Theorem 7.2 of \cite{Higman}. If we use \cite{Nek} to represent the elements of $V_n$ as unitaries $g\in \mathcal{O}_n$ in the algebraic form $g=\sum_{i=1}^m s_{a_i}s_{b_i}^\ast$ as in \eqref{eq:E}, then it is straightforward to see that $V_{k(n-1)+1}$ into $V_{n}$ when using the restriction of the morphism $\iota$ of Lemma \ref{thm:OninO2gen} to unitaries of $\mathcal{O}_n$ in such algebraic form.

 We further remark that given any $n,m \geq 2$, we have that there is an embedding from $V_n$ to $V_m$. This is so because Theorem \ref{thm:EmbFinT} implies that $V_n$ embedds in $V_2$ and then a recent result \cite{birget} shows that $V_2$ can be embedded in $V_m$.

Given any $n,m \geq 2$ it is also known that there are quasi-isometric embeddings from $F_n$ to $F_m$, as proven in \cite{embeddingsF}. Similarly, in \cite{embeddingsT}, it is proven that there is a quasi-isometric embedding from $T_{k(n-1) +1}$ to $T_n$, but that there are no embeddings from $T_2$ to $T_n$. Taking into account the foregoing considerations, one may ask which results applies to Cuntz algebras. For that we remark that, as stated in \cite{kawamura}, if there is a non trivial embedding from $\mathcal{O}_m$ to $\mathcal{O}_{n}$, then $m$ must be equal to $k(n-1)+1$, for some $k \geq 1$.

\subsection{Higman-Thompson groups as tables}
\label{HTtables}

	Let $n \in \mathbb{N}$ and fix the alphabet $X = \{1,\ldots,n\}$. In this section, we will clarify the relation between the Higman-Thompson groups as piecewise linear maps and as tables. We start by defining the following map.
	
	\begin{definition}
		\label{def:phi(x)}
		Let $\phi:X^*\rightarrow \mathcal{P}$ be:
		\[ \phi(u) = \phi(u_1u_2\ldots u_m) = \left[\sum_{i=1}^{m}\frac{u_i-1}{n^i},\sum_{i=1}^{m}\frac{u_i-1}{n^i} + \frac{1}{n^m}\right[ \]
		where $\mathcal{P} = \left\{\left[ \displaystyle \frac{a}{n^k},\frac{a+1}{n^k}\right[: a,k \in \mathbb{N}, 0 \leq a < n^k\right\}$.
	\end{definition}
We start by proving the following lemma.

	\begin{lemma}
		\label{lma:phi_bijective}
${}_{}$\par
\begin{enumerate}
\item $\phi$ is bijective.
\item Let $a,b \in X^*$. Then, $\phi(b) \subset \phi(a)$ if and only if $a$ is a prefix of $b$.
\end{enumerate}
	\end{lemma}
	
	\begin{proof}
		(1)\ First, we prove $\phi$ is injective. Let $a,b \in X^*$, and suppose $\phi(a) = \phi(b)$.  We can write $a$ and $b$ as $a = a_1a_2...a_k$ and $b = b_1...b_m$ for some $k,m$. From the system of equations	
		\[ \left\{
		\begin{array}{l}
		\displaystyle \sum_{i=1}^{k}\frac{a_i-1}{n_i} = \sum_{i=1}^{m}\frac{b_i-1}{n_i}\\
		\displaystyle \sum_{i=1}^{k}\frac{a_i-1}{n_i} + \frac{1}{n^k} = \sum_{i=1}^{m}\frac{b_i-1}{n_i} + \frac{1}{n^m}
		\end{array}
		\right. \]
	we get  $\frac{1}{n^k} = \frac{1}{n^m}$ and therefore $k = m$, so that $a$ and $b$ have the same length. Thus
		\[
		\begin{array}{lll} \displaystyle
		 \sum_{i=1}^{k}\frac{a_i-1}{n_i} = \sum_{i=1}^{k}\frac{b_i-1}{n_i}& \displaystyle \Rightarrow& \displaystyle  \sum_{i=1}^{k}(a_i-1)n^{k-i}= \sum_{i=1}^{k} (b_i-1)n^{k-i}\\
		& \Rightarrow & \displaystyle ((a_1-1)\ldots (a_k-1))_n = ((b_1-1)\ldots (b_k-1))_n,
		\end{array}     \]
		where $(x_1x_2\ldots x_k)_n$ is the numeric representation of number $x = x_1x_2...x_n$ (with $x_1 >0$, $x_i \in \{0,1,2,...,(n-1)\}$) in base $n$. Since both words give the same representation in base $n$ and this representation is unique except for zeroes in the beginning, we conclude that $a = b$, since $a$ and $b$ have the same length.
		
		Now, let $[\alpha,\beta[ \in \mathcal{P}$. By the definition of $\mathcal{P}$, we know that there exist $k,l \in \mathbb{N}$ such that $\alpha = kn^{-l}$. Let $k = (k_1k_2\ldots k_l)_n = \sum_{i=1}^{l}k_in^{l-i}$. Then:
	\begin{eqnarray*}
	  \phi((k_1+1)(k_2+1)\ldots(k_l+1)) = \left[\sum_{i=1}^{l}\frac{k_i}{n^i}, \sum_{i=1}^{l}\frac{k_i}{n^i} + \frac{1}{n^l}\right[  \\
	 = \left[\displaystyle \frac{\sum_{i=1}^{l}k_in^{l-i}}{n^l}, \frac{\sum_{i=1}^{l}k_in^{l-i}}{n^l} + \frac{1}{n^l}\right[ = [\alpha,\beta[.
	\end{eqnarray*}


(2)\ Let $a$ be a prefix of $b$. Then we can write $a = a_1a_2..a_m$, $b = a_1a_2\ldots a_mb_{m+1}\ldots b_{k}$. The result then follows from:
		
		\begin{eqnarray*}
		  \sum_{i=1}^{m}\frac{a_i-1}{n^i} &\leq& \sum_{i=1}^{m}\frac{a_i-1}{n^i} + \sum_{i=m+1}^{k}\frac{b_i-1}{n^i} \leq \sum_{i=1}^{m}\frac{a_i-1}{n^i} + \left(\sum_{i=m+1}^{k}\frac{b_i-1}{n^i} + \frac{1}{n^k}\right)\\
		  &\leq& \sum_{i=1}^{m}\frac{a_i-1}{n^i} + \frac{1}{n^m}
		\end{eqnarray*}
		where the last inequality follows from the fact that  for all $i$, $0 \leq (b_i-1) \leq n-1$ and also
		\begin{eqnarray*}
		  \sum_{i=m+1}^{k}\frac{b_i-1}{n^i} + \frac{1}{n^k} \leq \sum_{i=m+1}^{k} \frac{n-1}{n^{i}} + \frac{1}{n^k} &=&  \sum_{i=m+1}^{k} \left( \frac{1}{n^{i-1}} - \frac{1}{n^{i}}\right) + \frac{1}{n^k} \\
		   &=&  \left(\frac{1}{n^{m}} - \frac{1}{n^k} \right) + \frac{1}{n^k} = \frac{1}{n^m}.
		\end{eqnarray*}
		
		We now prove the converse. Let $a = a_1\ldots a_m$, $b = b_1\ldots b_k$, and suppose that $\phi(b) \subset \phi(a)$, but that $a$ is not a prefix of $b$. Since $\phi(b) \subset \phi(a)$, we have that $n^{-k} \leq n^{-m}$, and thus $m \leq k$.
		
		Notice that if we denote by $X_l$ the set of words of size $l$, then, for all $l$, $\phi(X_l)$ is a partition of $[0,1[$. This follows from the fact that the restriction $\phi:X_l\rightarrow\{[\frac{a}{n^l},\frac{a+1}{n^l}[:0\leq a\leq n^l\}$ is also a bijection.
		
		Because $m \leq k$, if $a$ is not a prefix of $b$, then there is a word $u \neq a$ of size $m$, such that $u$ is the prefix of $b$. But then, by the only if proof above, $\phi(b)\subset\phi(u)$, and therefore $\phi(u)\cap\phi(a) \neq \emptyset$. A contradiction! Thus $a$ must be a prefix of $b$.
	\end{proof}
	
	Given $g \in V_n$, we can write $g$ as $\{(a_1,b_1),(a_2,b_2),\ldots,(a_k,b_k)\}$, where $(a_i,b_i) \in \mathcal{P}^2$, $\bigcup_{i=1}^k a_i = \bigcup_{i=1}^k b_i = [0,1[$ and $a_i\cap a_j = b_i \cap b_j = \emptyset$ if $i \neq j$. Therefore the elements of $V_n$ can be written as maps between two partitions $A,B$ of $[0,1[$, such that $A,B \subset \mathcal{P}$.  Thus, given $a_i = [\alpha_i,\alpha_i+ n^{-k}[$, $b_i = [\beta_i,\beta_i+n^{-l}[$ and $x \in a_i$, the map
\begin{equation}
g(x) = \beta_i + n^{k-l}(x-\alpha_i)
\label{eqtablePL}
\end{equation}
is such that $g\in V_n$ as in Definition \ref{def:VTF}.
Given an element of the Higman-Thompson group $V_n$, we can write $\Psi(g)$ explicitly as
	\begin{eqnarray*}
	  \Psi(g) &=& \Psi(\{(a_1,b_1),(a_2,b_2),\ldots,(a_k,b_k)\}) \\
	   &=& s_{\phi^{-1}(a_1)}s^*_{\phi^{-1}(b_1)} + s_{\phi^{-1}(a_2)}s^*_{\phi^{-1}(b_2)} + \ldots + s_{\phi^{-1}(a_k)}s^*_{\phi^{-1}(b_k)}.
	\end{eqnarray*}
	 For this to be well defined we need to show that $\{\phi^{-1}(a_1), \ldots,\phi^{-1}(a_k) \}$ and $\{\phi^{-1}(b_1), \ldots,\phi^{-1}(b_k) \} $ are admissible sets. This is the aim of the next proposition.
	
	\begin{proposition}
		\label{prp:sets<->words}
		Let  $A = \{a_1,\ldots a_k\}$ be a partition of $[0,1[$, such that $a_i \in \mathcal{P}$ for all $i$. Then, for all $i\neq j$:
		\[ \bigcup_{m=1}^k \phi^{-1}(a_m)X^{\omega}  = X^{\omega} \hspace*{15pt} \mbox{and} \hspace*{15pt}  \phi^{-1}(a_i)X^\omega \cap\phi^{-1}(a_j)X^\omega = \emptyset.\]
		
	\end{proposition}
	
	\begin{proof}
		Let $x \in X^\omega$. Suppose there is no $i$ such that $\phi^{-1}(a_i)$ is a prefix of $x$. For any $l\in \mathbb{N}$, we can write $x$ as $x_1x_2\ldots x_lu$, where $u \in X^\omega$. Let $w = \sum_{i=1}^{l}\frac{x_i-1}{n^i}$. By definition, $\phi(x_1x_2\ldots x_l) = [w,w+n^{-l}[ $. Since $A$ is a partition of $[0,1[$, we know there is an $a_j \in A$ such that $a_j = [a,b[$, and $w \in a_j$. Since $l$ is arbitrary, we can choose $l$ such that $n^{-l}<b-w$ which implies $\phi(x_1x_2\ldots x_l) \subset a_j$. By Lemma \ref{lma:phi_bijective} (2), this implies that $\phi^{-1}(a_j)$ is a prefix of $x_1x_2\ldots x_l$, and thus of $x$. A contradiction. It is proved that
\[ \bigcup_{m=1}^k \phi^{-1}(a_m)X^{\omega}  = X^{\omega}.\]
		
		Now, suppose there are $i \neq j$ such that $\phi^{-1}(a_i)X^\omega \cap\phi^{-1}(a_j)X^\omega = \{x\}$ for some $x \in X^\omega$. Then, both $\phi^{-1}(a_j)$ and $\phi^{-1}(a_i)$ will be prefixes of $x$. By Lemma \ref{lma:phi_bijective} (2), this implies that there exists an $l$ big enough such that $\phi(x_1x_2\ldots x_l) \subset a_i$ and $\phi(x_1x_2\ldots x_l) \subset a_j$, a contradiction, since $A$ is a partition. Thus, $\phi^{-1}(a_i)X^\omega \cap\phi^{-1}(a_j)X^\omega = \emptyset$.
\end{proof}

\section{Representations of $V_n$ from the Cuntz algebra $\mathcal{O}_n$}
	\label{Representations of $V_n$ in $H$}


	Using $\Psi$ as in \eqref{eq:Psi}, we can now associate a representation of the Higman-Thompson group $V_n$ from every representation of the Cuntz algebra $\mathcal{O}_n$, acting on the same Hilbert space $H$.
	\begin{definition}
		\label{def:rho(g)}
		Given a Hilbert space $H$ and a representation $\pi:\mathcal{O}_n \rightarrow B(H)$, we define the map $\rho_{\pi}:V_n \rightarrow B(H)$ as
		$$\rho_{\pi}(g) = (\pi \circ \Psi)(g).$$ 	
	\end{definition}
	We then have the following theorem.
	\begin{theorem}
	\label{thm:rhorep}
		$\rho_{\pi}$ is a unitary representation of $V_n$ in $H$.
	\end{theorem}
	
	\begin{proof}
		Since both $\pi$ and $\Psi$ are homomorphisms, we conclude that $\rho_{\pi}$ is a representation of $V_n$ in $H$. In order to prove that $\rho_{\pi}$ is unitary, we must show that
		$$ \rho_{\pi}(g^{-1}) = (\rho_{\pi}(g))^*.$$
		Using the fact that given a map $g = \{(a_1,b_1),(a_2,b_2),\ldots,(a_k,b_k)\} $ we can write $g^{-1}$ as $$g^{-1} = \{(b_1,a_1),(b_2,a_2),\ldots,(b_k,a_k)\}$$ we automatically have
		\begin{eqnarray*}
		  \rho_{\pi}(g^{-1}) &=& s_{\phi^{-1}(b_1)}s^*_{\phi^{-1}(a_1)} + \ldots + s_{\phi^{-1}(b_k)}s^*_{\phi^{-1}(a_k)}\\ &=& (s_{\phi^{-1}(a_1)}s^*_{\phi^{-1}(b_1)} + \ldots + s_{\phi^{-1}(a_k)}s^*_{\phi^{-1}(b_k)})^*
		   = (\rho_{\pi}(g))^*.
		\end{eqnarray*}
	\end{proof}

\subsection{Representations $\rho_x$ of $V_n$ on the orbits of the interval map $f(x)=nx$ (mod 1)}
\label{sect:rhoxorb}
	
	In this subsection, we will study the representations of $V_n$ in a specific set of Hilbert spaces. Fix $x \in [0,1[$ and define its (generalized) orbit as
\begin{equation}
\text{orb}(x) := \{f^z(x): z \in \mathbb{Z}\}
\label{eq:orb}
\end{equation}
where $f(y) = ny \mod 1$ with $y\in [0,1[$. Sometimes we will use the notation $\orb(x)$ insatead of $\text{orb}(x)$. Note that orb$(x)$ is the forward and backward orbit of $x$, thus $f(x)=nx$ (mod 1), $f^{-1}(x) = \{ \frac{x}{n}, \frac{x+1}{n}, \ldots, \frac{x+ (n-1)}{n}\}$, $$f^{-2}(x) = \bigcup_{y \in f^{-1}(x)}f^{-1}(y)$$ and so on.

Let $\sim$ be a binary relation in $[0,1[$ such that $y \sim x$ if and only if $y \in \text{orb}(x)$. This is equivalent to existing $p,k \in \mathbb{N}$ such that $f^p(y) = f^k(x)$. One can easily prove that $\sim$ is an equivalence relation whose equivalence classes are the different orbits. Denote by $H_x = l^2(\text{orb}(x))$. The set $\{\delta_y: y \in \text{orb}(x)\}$, where $\delta_y:\text{orb}(x)\rightarrow\mathbb{R}$ is the map
   \[  \delta_y(z) = \left\{
   \begin{array}{ll}
   1 & y = z\\
   0 & y \neq z
   \end{array}
   \right.  \]
     is an orthonormal basis of $H_x$. Since $\text{orb}(x)$ is a countable union of countable sets, it is a countable set. Therefore, $H_x$ is a separable Hilbert space.
	\begin{definition}
		\label{ref:Si}
		Let $i\in\{1,....,n\}$. Define $S_i \in B(H_x)$ first on the vector basis
		\[ S_i\delta_y = \delta_{\frac{y+(i-1)}{n}}\]
then extend it to the linear span and finally extended to $H_x$ by density of the basis and continuity of $S_i$ on the algebraic span of $\{\delta_y: y\in \text{orb(x)}\}$.
	\end{definition}
For every $i$, we easily check that the adjoint of $S_i$ is as follows
	\[ S^*_i \delta_y = \left\{
	\begin{array}{ll}
	\delta_{yn-(i-1)} & y \in [\frac{i-1}{n},\frac{i}{n}[\\
	0 & y \notin 	[\frac{i-1}{n},\frac{i}{n}[
	\end{array}
	\right.  \]
	thus $S_i$ is an isometry $S_i^\ast S_i=1$. One can verify that the operators $S_i$ ($i=1,...,n$) satisfy the Cuntz relations. We therefore have the following result.
	\begin{theorem}
	\label{thm:pi}
	The map $\pi_x:\mathcal{O}_n \rightarrow B(H_x)$ such that $\pi_x(s_i) = S_i$ is a representation of $\mathcal{O}_n$ in $H_x$.
	\end{theorem}

\begin{notation}
If confusion arises, we will denote:
\begin{enumerate}
  \item the representations $\pi_x$ introduced in Theorem \ref{thm:pi} by $\pixn$;
  \item the representation $\rho_{\pi_x}$ of $V_n$ by $\rho_x$ or $\rhoxn$ (see Definition \ref{def:rho(g)});
      \item the Hilbert space $H_x=\ell^2(\hbox{orb}(x))$ by $\hxn$.
\end{enumerate}
\label{not1}
\end{notation}

	 The next result is the main result of this section: its shows that the action of $V_n$ on the set $\{\delta_y: y\in\hbox{orb}(x)\}$ is well defined and the underlying representation of $V_n$ on $H_x$ coincides with  $\rho_x$.

	\begin{theorem}
	\label{thm:bigs4}
	Given $g \in V_n$ and $y \in \text{orb}(x)$, we have
	\begin{enumerate}
	    \item $g(y) \in \text{orb}(x)$,
	    \item $\rho_x(g)(\delta_y) = \delta_{g(y)}$.
	\end{enumerate}
	
	\end{theorem}
	
	This theorem is going to be proved in a series of lemmas. The first claim is Lemma \ref{lma:g(y)inorb}. The second claim is Lemma \ref{lma:rho(g)y=g(y)}.
	
Our first goal is to prove Lemma \ref{lma:g(y)inorb}. The idea is that for any $y$, $g(y)$ will be of the form $an^k +yn^m$. Therefore, we need to show that multiplying an element of $\text{orb}(x)$ by powers of $n$ preserves the orbit, and that for each $y \in \text{orb}(x)$, there is an element of the form $an^k +yn^m$ in $\text{orb}(x)$. We start by proving that powers of $n$ preserve the orbit.
	
	\begin{lemma}
		\label{lma:kninorb}
		Let $k \in \text{orb}(x)$. Then, given $m \in \mathbb{Z}$, $kn^m \mod 1 \in \text{orb}(x)$
	\end{lemma}
	
	\begin{proof}
		The result is obvious for $m \leq 0$. For $m>0$, we prove the result by induction on $m$.
		We know that, for any given $k \in \text{orb}(x)$, we have that $(nk) \mod 1 \in \text{orb}(x)$. The result is thus true for $m = 1$.
		
		Suppose the result is true for $m$.  Notice that $n^mk = (n^mk \mod 1) + \lfloor n^mk \rfloor$. Then, since
	$ n^{m+1}k \mod 1 = n(n^mk) \mod 1$
\[ = n((n^mk \mod 1) + \lfloor n^mk \rfloor) \mod 1 = n(n^mk \mod 1) \hbox{mod}\ 1 \]
		and by assumption, $(n^mk \mod 1) \in \text{orb}(x)$, we conclude that $n^{m+1}k \mod 1 \in \text{orb}(x)$.
	\end{proof}
	
	We now show that for any $a \in \mathbb{N}$, and $y \in \text{orb}(x)$, there is an element of the form $an^k +yn^m$ in $\text{orb}(x)$.
	
	\begin{lemma}
		\label{lma:kn+ainorb}
		Let $k \in \text{orb}(x)$. Then for any $a \in \mathbb{N}$ such that the number of digits in base $n$ of $a$ is $m$, we have that $kn^{-m} + an^{-m} \in \text{orb}(x)$.
	\end{lemma}
	
	\begin{proof}
		Let $(a_1a_n\ldots a_m)_n$ be the representation of $a$ in base $n$. We claim that $a_1 > 0$ and $a_i \in \{0,1,2,\ldots,(n-1)\}$.
		
		We prove the result by induction on $m$. If $m = 1$, we have $a \in \{1,2,\ldots,(n-1)\}$. Since, given $y \in \text{orb}(x)$, and $h \in \{0,1,2,\ldots,(n-1)\}$, $(y+h)/n \in \text{orb}(x)$ the result follows for $m=1$.
		
		Now, let $j$ be the first number such that $a_j > 0$ and $j>1$. We have $a = a_1n^{m-1} + (a_ja_{j+1}\ldots a_m)_n$. By the induction hypothesis, we have $kn^{j-m} + (a-a_1n^{m-1})n^{j-m} \in \text{orb}(x)$. Multiplying by $n^{-(j-1)}$ and using Lemma \ref{lma:kninorb} we obtain $n^{-(j-1)}(kn^{j-m} + (a-a_1n^{m-1})n^{j-m}) \in \text{orb}(x)$ and thus $kn^{-(m-1)} + (a-a_1n^{m-1})n^{-(m-1)} \in \text{orb}(x)$. Then	
		\[ \frac{(kn^{-(m-1)} + (a-a_1n^{m-1})n^{-(m-1)}) + a_1}{n} \in \text{orb}(x) \Rightarrow kn^{-m} + an^{-m} - a_1n^{-1} + a_1n^{-1} \]
which is in \text{orb}(x)
		and therefore $kn^{-m} + an^{-m}$.
	\end{proof}
	
	We can now prove Lemma \ref{lma:g(y)inorb}. Observe that, given $y \in [0,1[$, we have $g(y) = \alpha n^b + yn^c$, with $\alpha\in \mathbb{Z}$. Also, since $g(k) \in [0,1[$, we have $g(y) = \alpha n^b + yn^c \mod 1$. As $\alpha n^b + yn^c \mod 1 = (\alpha n^b \mod 1) + yn^c \mod 1$, we can write $g(y) = an^b + yn^c \mod 1$, where $an^b = (\alpha n^b \mod 1) \in [0,1[$.
	
	\begin{lemma}
		\label{lma:g(y)inorb}
	Let $g \in V_n$, $f(x)= nx \mod 1$ and $\text{orb}(x) = \bigcup_{m \in \mathbb{Z}}\{f^m(x)\}$.  Then, for all $y \in \text{orb}(x)$, $g(y) \in \text{orb}(x)$.
	\end{lemma}
	
	\begin{proof}
		Let $y \in \text{orb}(x)$. We have that $g(y) = an^b + yn^c \mod 1$. Let $n^{m-1} < a <n^m$. From Lemma \ref{lma:kn+ainorb}, we have $kn^{-m} + an^{-m}  \in \text{orb}(x)$ for any $k  \in \text{orb}(x)$.
		By Lemma \ref{lma:kninorb}, $yn^{c-b} \mod 1 \in \text{orb}(x)$. Replacing $k$ with $yn^{c-b} \mod 1$ we obtain $yn^{(c-b)-m} + an^{-m} \mod 1 \in \text{orb}(x)$. Using Lemma \ref{lma:kninorb}, we conclude that $n^{b+m}(yn^{(c-b)-m} + an^{-m}) \mod 1 = an^b + yn^c \mod 1 = g(y) \in \text{orb}(x)$.
	\end{proof}
	
	The goal of the next sequence of lemmas is to prove Lemma \ref{lma:rho(g)y=g(y)}. We start by giving an intuitive idea of the proof. We have that $\rho_x(g) = S_{\phi^{-1}(a_1)}S^*_{\phi^{-1}(b_1)} + \ldots + S_{\phi^{-1}(a_k)}S^*_{\phi^{-1}(b_k)}$. We know that there is one and only one $i$ such that $y \in b_i$. We will show that $S^*_{\phi^{-1}(b_j)}\delta_y = 0$ if $j \neq i$. Then, we will show that $S_{\phi^{-1}(a_i)}S^*_{\phi^{-1}(b_i)}\delta_y = \delta_{g(y)}$.
	
	\begin{lemma}
		\label{lma:SuDy}
		Let $y \in \text{orb}(x) $ and $u \in X^*$, such that $u = u_1\ldots u_k$. Then $S_u\delta_{y}= S_{u_1}\ldots S_{u_k}\delta_y = \delta_{a}$, where
		\[ a = yn^{-k} + \sum_{i=1}^{k}\frac{u_i-1}{n^i}. \]
		
	\end{lemma}
	
	\begin{proof}
		We prove the result by induction in $k$ . For $k=1$, we have that $S_u\delta_y = S_{u_1}\delta_y = \delta_{\frac{y+(u_1-1)}{n}}$.
		
		Suppose the result is true for $k-1$. Then $S_u\delta_{y} = S_{u_1}(S_{u_2}\ldots S_{u_k}\delta_{y}) = S_{u_1}\delta_{a_{k-1}} = \delta_{a}$, where	
		\[ a_{k-1} = yn^{-(k-1)} + \sum_{i=2}^{k}\frac{u_i-1}{n^{i-1}}, \]
		from which we obtain
		\[  a = \frac{a_{k-1}+(u_1-1)}{n} = \frac{yn^{-(k-1)} + \displaystyle \sum_{i=2}^{k}\frac{u_i-1}{n^{i-1}}+(u_1-1)}{n} =  yn^{-k} + \sum_{i=1}^{k}\frac{u_i-1}{n^i}.  \]
		
		The result follows.
	\end{proof}

 Next, given a word $v$, we want to see how $S_v^*$ acts on $H_x$. Yet, in order to do this, we will first need Lemma \ref{lma:b<->y}.
	
	\begin{lemma}
		\label{lma:b<->y}
		Let $y \in \text{orb}(x) $ and $v \in X^*$, such that $v = v_1\ldots v_mv_{m+1}$. Then, $y \in \phi(v)$ if and only if $b \in \phi(v_{m+1})$, where
		\[ b = n^m\left(y - \sum_{i=1}^{m}\frac{v_i-1}{n^i}\right). \]
		\end{lemma}
	\begin{proof}
The proof follows from the definition of $\phi$: $b \in \phi(v_{m+1}) \Leftrightarrow $
		\[\frac{v_{m+1}-1}{n} \leq  n^m\left(y - \sum_{i=1}^{m}\frac{v_i-1}{n^i}\right)  < \frac{v_{m+1}}{n} \Leftrightarrow \frac{v_{m+1}-1}{n^{m+1}} \leq  \left(y - \sum_{i=1}^{m}\frac{v_i-1}{n^i}\right)  <\frac{v_{m+1}}{n^{m+1}} \]
		\[ \Leftrightarrow \sum_{i=1}^{m+1}\frac{v_i-1}{n^i} \leq y < \sum_{i=1}^{m+1}\frac{v_i-1}{n^i} +\frac{1}{n^{m+1}} \Leftrightarrow y \in \phi(v).  \]
	\end{proof}
	
	\begin{lemma}
		\label{lma:S*vDy}	
		Let $y \in \text{orb}(x) $ and $v \in X^*$, such that $v = v_1\ldots v_m$ and $y \in \phi(v)$. Then $S^*_v\delta_{y}=\delta_{b}$, where
		\[ b = n^{m}\left(y - \sum_{i=1}^{m}\frac{v_i-1}{n^i}\right). \]
		
	\end{lemma}
	
	\begin{proof}
		We will prove the result by induction on $m$. For $m = 1$, we can write $v = v_1 = i$. If $y \in \phi(v) = [\frac{i-1}{n},\frac{i}{n}[$, we have that $S^*_v\delta_y = \delta_{yn-(i-1)}$.
		
		Now, suppose that the result is true for $m$. We can use Lemma \ref{lma:phi_bijective} (2) to conclude that, since $v_1\ldots v_{m}$ ia a prefix of $v$, we have that $\phi(v)\subset\phi(v_1\ldots v_{m})$, and thus $y \in \phi(v_1\ldots v_{m})$. We can thus apply the induction hypothesis to conclude that $S^*_v\delta_{y} = S^*_{v_{m+1}}(S^*_{v_{m}}\ldots S^*_{v_1}\delta_{y}) = S^*_{v_{m+1}}\delta_{b_{m}}$ where 	
		\[ b_{m} = n^{m}\left(y - \sum_{i=1}^{m}\frac{v_i-1}{n^i}\right). \]
		By Lemma \ref{lma:b<->y}, $b_{m} \in \phi(v_{m+1})$ and thus $S^*_{v_{m+1}}\delta_{b_{m}} = \delta_{n(b_m)-(v_{m+1}-1)} = \delta_{b}$ where
		\[ b = n^{m+1}\left(y - \sum_{i=1}^{m}\frac{v_i-1}{n^i}\right) - (v_{m+1}-1) = n^{m+1}\left(y - \sum_{i=1}^{m+1}\frac{v_i-1}{n^i}\right). \]
		
	\end{proof}

    Knowing how $S_v^*$ acts on $H_x$, we can now show that if $y \notin \phi(v)$, $S_v^*\delta_y = 0$.
    	
	\begin{lemma}
		\label{S*Dy=0}
		Let $v \in X^*$. If $y \notin \phi(v)$ then $S^*_v\delta_y = 0$.
	\end{lemma}
	
	\begin{proof}
		
		Let $v \in X^*$, then we can write $v = v_1v_2\ldots v_m$. We will prove by induction on $m$ that if $y \notin \phi(v)$, then $S^*_v\delta_y = 0$. For $m = 1$, we have that $v \in \{1,2,..,n\}$, and thus $\phi(v) = [\frac{v-1}{n}, \frac{v}{n}[$. Then $y \notin \phi(v)$ is equivalent to $y \notin [\frac{v-1}{n}, \frac{v}{n}[$ and thus $S^*_v\delta_y = 0$.
		
		Now, suppose the result is true for $m$ and let $v = v_1v_2\ldots v_{m+1}$. Suppose that $y \notin \phi(v)$. We have that either $y \in \phi(v_1v_2\ldots v_m)$ or $y \notin \phi(v_1v_2\ldots v_m)$. In the latter case, since $S^*_v\delta_y = S^*_{v_{m+1}}(S^*_{v_m}\ldots S^*_{v_2}S^*_{v_1}\delta_y)$, we have by the induction hypothesis that $S^*_v\delta_y = 0$. In the former case, we can apply Lemma \ref{lma:S*vDy} to conclude that $S^*_v\delta_y = S^*_{v_{m+1}}\delta_b$. By Lemma \ref{lma:b<->y}, $y\notin \phi(v)$ implies that $b \notin \phi(v_{m+1})$. Thus $S^*_v\delta_y = S^*_{v_{m+1}}\delta_b = 0$.
	\end{proof}
	
	The next lemma completes the proof of Theorem \ref{thm:bigs4}.
	
	\begin{lemma}
		\label{lma:rho(g)y=g(y)}
		Let $y \in \text{orb}(x) $. Then, $\rho_x(g)\delta_{y} = \delta_{g(y)}$.
	\end{lemma}
	
	\begin{proof}
		We have that $\rho_x(g) = S_{c_1}S^*_{d_1} + S_{c_2}S^*_{d_2} \ldots + S_{c_l}S^*_{d_l}$ where $c_i,d_i \in X^*$ for all $i$. Let $y \in \text{orb}(x)$. Since there is a bijective correspondence between the $d_i$'s, and a partition of $[0,1[$, we have that there is a unique $j$ such that $y \in \phi(d_j)$. By Lemma \ref{S*Dy=0}, we conclude that
		$$\rho_x(g)\delta_y = (S_{c_1}S^*_{d_1} + S_{c_2}S^*_{d_2} \ldots + S_{c_l}S^*_{d_l})\delta_y  = S_{c_j}S^*_{d_j}\delta_y.$$
		Let us denote $u = c_j$, and $v = d_j$. By Lemma \ref{lma:S*vDy}, $S^*_v\delta_y = \delta_{b}$ where
		$$b =  n^m\left(y - \sum_{i=1}^{m}\frac{v_i-1}{n^i}\right)$$
		and by Lemma \ref{lma:SuDy}, $S_u\delta_{b} = \delta_{g(y)}$, where
		$$g(y) = bn^{-k} + \sum_{i=1}^{k}\frac{u_i-1}{n^i} = n^{m-k}\left(y - \sum_{i=1}^{m}\frac{v_i-1}{n^i}\right) + \sum_{i=1}^{k}\frac{u_i-1}{n^i}.$$
		A linear transformation mapping $\phi(u)$ to $\phi(v)$.
	\end{proof}

	\subsection{Cuntz algebra embeddings and Higman-Thompson representations}
	\label{Representations of $V_n$ and Embeddings}
	
Let $n\in\mathbb{N}$ and $x \in [0,1[$. Recall the definition of orbit $\text{orb}_n(x)$ from \eqref{eq:orb} and let $\hxn=l^2(\text{orb}_n(x))$.  Our goal in this section is to study the family $\{\rho_x\}_{x\in [0,1[}$ of representations of the Higman-Thompson group $V_n$ introduced in the previous one.
	
We start by observing that for all $n>1$, the Cuntz algebra $\mathcal{O}_n$ is simple, as proven by Cuntz in \cite{cuntz}. Hence, every *-homomorphism of $\mathcal{O}_n$ is injective. In particular, the maps $\pi_x$ (see Theorem \ref{thm:pi}) and $\iota$ (see Lemma \ref{thm:OninO2gen}) are injective. Recall from Notation \ref{not1} that $\pixnk$ as the representation of $\mathcal{O}_{k(n-1)+1}$ in $\hxk$ such that the images of the generators are $\hat{S_1},....., \hat{S}_{k(n-1)+1}$. In this fixed context, we also use $S_1,....,S_n$ as the images of the generators of $\mathcal{O}_n$ in the representation $\pixn$ on $\hxn$ (see Notation \ref{not1}).
	
	\begin{definition}
		\label{def:iota'}
		Let $\iota': \pixnk(\mathcal{O}_{k(n-1)+1})\rightarrow B(\hxn)$ be defined as
		\[ \iota'(\hat{S_1}) = S_1^{k} \hspace*{30pt} \iota'(\hat{S}_{i(n-1)+j}) = S_1^{k-i}S_j\]	
		for $0\leq i < k$, $2 \leq j \leq n$.
	\end{definition}
	
Since all this maps are injective, we automatically get that $\iota'$ is injective. Furthermore, since $\iota$ is a *-homomorphism from $\mathcal{O}_{k(n-1)+1}$ to $\mathcal{O}_n$, it is continuous. Hence, due to the continuity of $\pixnk$ and $\pixn$,  $\iota'$ is continuous.  We obtain the following commutative diagram:
	
	\[ \begin{tikzcd}
	\mathcal{O}_{k(n-1)+1} \arrow[swap]{d}{\iota} & \arrow{l}{\quad {(\pixnk)}^{-1}}  \ \ \ \quad \pixnk(\mathcal{O}_{k(n-1)+1}) \arrow{d}{\iota'} \\%
	\mathcal{O}_n \arrow{r}{\pixn}& \pi'_x(\mathcal{O}_n).
	\end{tikzcd}
	\]

	\begin{proposition}
		\label{prp:YpathX}
		Let $y \in \text{orb}_m(x)$ for some $m\geq 2$. Then $\delta_y = T_1T_2\ldots T_k\delta_x$ for some $$T_1, T_2, \ldots ,T_k \in \{\hat{S_1},\ldots,\hat{S_m},\hat{S}^*_1,\ldots,\hat{S}^*_m\}.$$
	\end{proposition}
	
\begin{proof}
By definition, $\text{orb}_m(x) = \{f^z(x): z\in \mathbb{Z}, f(x) = mx \mod 1\}$. We can represent the orbit by a graph whose vertices are the elements of $\text{orb}_m(x)$, and $y-z$ if there is an $i \in \{0,\ldots,(m-1)\}$ such that $y = \frac{z+i}{m}$, or $z = \frac{y+i}{m}$. The graph must be connected, since it is inductively constructed from $x$ adding the vertices corresponding to the operations $y \mapsto \frac{y-i}{m}$ and $y \mapsto ym\mod 1$. Let us denote by $x_0, x_1, x_2, \ldots, x_k$ a path starting in $x_0 = x$ and ending on $x_k = y$. Given $x_j$ and $x_{j+1}$, we either have that $x_{j+1} = \frac{x_{j}+i}{m}$ or $x_j = \frac{x_{j+1}+i}{m}$. In the first case, we conclude that $\delta_{x_{j+1}} = \hat{S}_{i+1}\delta_{x_j}$. In the second case, we conclude that  $\delta_{x_{j+1}} = \hat{S}^*_{i+1}\delta_{x_j}$. Repeating the process $k$ times, we conclude that there exist $T_1, T_2, \ldots ,T_k \in \{\hat{S}_1,\ldots,\hat{S}_m,\hat{S}^*_1,\ldots,\hat{S}^*_m\}$ such that $\delta_y = T_1T_2\ldots T_k\delta_x$.
\end{proof}

We now use Proposition \ref{prp:YpathX} in the following definition.
	\begin{definition}
		\label{def:U}
		We define $U:\hxk\rightarrow \hxn$ as the map such that, given $\delta_y \in \hxk$
		\[ U(\delta_y) = U(T_1\ldots T_k\delta_x) = \iota'(T_1\ldots T_k)\delta_x. \]
	\end{definition}
	
	We start by showing that $U$ is well defined. Suppose that $\delta_y = T_1\ldots T_k\delta_x$ and $\delta_y = L_1\ldots L_l\delta_x$. Then, because $\iota'$ is a map, $\iota'(L_1\ldots L_l) = \iota'(T_1\ldots T_k)$, and thus $U( T_1\ldots T_k\delta_x) = U(L_1\ldots L_l\delta_x)$. Having defined $U$ on the basis of $\hxk$, we can extend it to the linear span and then by continuity and density to the whole space $\hxk$.
	
	We claim that $U:\hxk\rightarrow U(\hxk)$ is unitary. In fact, $U$ is injective and linear since $\iota'$ is injective and linear. Therefore, given $\delta_y, \delta_z \in \hxk$, we have
	\[ \langle U\delta_y, U\delta_z \rangle_{\hxn}  = \left\{
	\begin{array}{ll}
	1 & \text{; } U\delta_y = U\delta_z\\
	0 & \text{; } U\delta_y \neq U\delta_z
	\end{array}
	\right.  = \left\{
	\begin{array}{ll}
	1 & \text{; } \delta_y = \delta_z\\
	0 & \text{; } \delta_y \neq \delta_z
	\end{array}
	\right. = \langle \delta_y, \delta_z \rangle_{\hxk}. \]
	Hence, $U$ is unitary in $\hxk$ because it is unitary on the linear span of its basis. The claim follows.
	
	Let $\rhoxn:V_n\rightarrow B(\hxn)$ be the representation of $V_n$ in $\hxn$ (see Notation \ref{not1}). We want to study the relation between $\rhoxk$ and $\rhoxn(\E(g))$, the restriction of $\rhoxn$ to the elements of a subgroup of $V_n$ isomorphic to $V_{k(n-1)+1}$. (Recall that $\E$ is the embedding of $V_{k(n-1)+1}$ in $V_n$, introduced in Theorem \ref{thm:EmbFinT}.) The following theorem gives us the relation between these maps.
	
	\begin{theorem}
		\label{thm:r'E=s'r}
	For $g \in V_{k(n-1)+1}$ we have $\rhoxn(\E(g)) = \iota'(\rhoxk(g))$.
	\end{theorem}
	
	\begin{proof}
		We will represent $g$ as a table. From Theorem \ref{thm:EmbFinT} we have that
		\begin{eqnarray*}
	    \rhoxn(\E(g)) &=& \rhoxn\left(\E\left( \begin{bmatrix}
		a_1 &  \ldots & a_m \\
		b_1 &  \ldots & b_m
		\end{bmatrix} \right)\right) =  \rhoxn\left( \begin{bmatrix}
		f(a_1) & \ldots & f(a_m) \\
		f(b_1) & \ldots & f(b_m)
		\end{bmatrix} \right)\\
		&=& S_{f(a_1)} S^*_{f(b_1)} + \ldots + S_{f(a_m)} S^*_{f(b_m)}.
	\end{eqnarray*}
		
		On the other hand, from definition \ref{def:gamma,f}, we have
		\begin{eqnarray*}
		\iota'(\rhoxk(g)) = \iota'( \hat{S}_{a_1} \hat{S}^*_{b_1} + \ldots + \hat{S}_{a_m} \hat{S}^*_{b_m}) &=& \iota'(\hat{S}_{a_1}) \iota'(\hat{S}^*_{b_1}) + \ldots + \iota'(\hat{S}_{a_m}) \iota'(\hat{S}^*_{b_m})\\
		&=&  S_{f(a_1)} S^*_{f(b_1)} + \ldots + S_{f(a_m)} S^*_{f(b_m)}.
		\end{eqnarray*}
	\end{proof}
	
The operator $U$ turns out to be an unexpectedly powerful tool.  In fact, Theorem \ref{thm:r'E=s'r} implies that for any $g \in V_{k(n-1)+1}$ and some $U(\xi) \in U(\hxk)$,
\begin{equation}\label{eq:inv}
\rhoxn(\E(g))U(\xi) = \iota'(\rhoxk(g))U(\xi) = U(\rhoxk(g)(\xi))
\end{equation}
where $\xi \in \hxk$.
Therefore, $U(\hxk)$ is a proper subset of $\hxn$ and invariant under $\rhoxn(\E(g))$ for any $g$. Therefore, the representation $\rhoxn\circ \E$ of $V_{k(n-1) + 1}$ in $\hxn$ is not irreducible. In fact, we will prove that $\rhoxn\circ \E$ is an irreducible representation of $V_{k(n-1) + 1}$ on $U(\hxk)$ instead. In order to do this, we will need Theorem \ref{thm:Upi=pi'U} which will allow us to relate what happens in $\hxk$ to what happens in $U(\hxk)$.

We can adapt the proof in Eq.\ \eqref{eq:inv} and check that $(\iota^\prime \circ \pi_x)(a)(U(\xi)) \in U(\hxk)$ for any $a\in \mathcal{O}_{k(n-1)+1}$ and $\xi \in \hxk$. This means that besides the representation  $\pi_x:\mathcal{O}_{k(n-1)+1}\rightarrow B(\hxk)$ we also have a well defined representation $\iota' \circ \pi_x:\mathcal{O}_{k(n-1)+1}\rightarrow B(U(\hxk))$. We now relate these two families of Cuntz algebras representations of the Cuntz algebra $\mathcal{O}_{k(n-1)+1}$.
	
	\begin{theorem}
		\label{thm:Upi=pi'U}
		Let $x,y \in [0,1[$. Then:
\begin{enumerate}
  \item $\pi_x$ and $\iota' \circ \pi_x$ are unitarily equivalent;
  \item $\pi_x$ is unitarily equivalent to $\pi_y$ if and only if $\iota'\circ\pi_x$ is unitarily equivalent to $\iota'\circ\pi_y$;
  \item $\pi_x$ is irreducible in $\hxk$ if and only if $\iota'\circ\pi_x$ is irreducible in $U(\hxk)$.
\end{enumerate}

	\end{theorem}
	
	\begin{proof}
	(1)\ Since we already proved that $U:\hxk\rightarrow U(\hxk)$ is unitary, it remains to show that given $a\in \mathcal{O}_{k(n-1)+1}$, $y \in \text{orb}_{k(n-1)+1}(x)$, we have $U(\pi_x(a)\delta_y) = (\iota'(\pi_x(a))U)\delta_y$, where $U$ was introduced in Definition  \ref{def:U}. Thus
		\[
		\begin{array}{llll}
		(\iota'(\pi_x(a))U)\delta_y&=& \iota'(\pi_x(a))(U(T_1\ldots T_t \delta_x)) & (\text{by Proposition } \ref{prp:YpathX}) \\
		&=& \iota'(\pi_x(a))\iota'(T_1\ldots T_t)\delta_x  & (\text{by Definition } \ref{def:U})\\
		&=& \iota'(\pi_x(a)T_1\ldots T_t)\delta_x& (\text{Since
		$\iota'$  is linear})\\
		&=& U(\pi_x(a)T_1\ldots T_t\delta_x)& (\text{by Definition } \ref{def:U})\\
		&=& U(\pi_x(a)\delta_y).&
		\end{array}
		\]

(2)\ Suppose that $\pi_x \sim \pi_y$, that is, that $\pi_x$  and $\pi_y$ are unitarily equivalent. Then, there is a unitary $K:\hxk \rightarrow \hyk$ such that $K\pi_x(a) = \pi_y(a)K$, for any $a \in \mathcal{O}_{k(n-1)+1}$.
		Using part (1) we obtain that $\iota'(\pi_y(a))(U_2KU^*_1) = (U_2KU^*_1) \iota'(\pi_x(a))$ which corresponds to the commutative diagram	
		\[ \begin{tikzcd}
		U_1(\hxk) \arrow{r}{U^*_1} \arrow{d}{\iota'(\pi_x(a))} & \hxk \arrow{r}{K} \arrow{d}{\pi_x(a)} & \hyk \arrow{r}{U_2} \arrow{d}{\pi_y(a)}  & U_2(\hyk) \arrow{d}{\iota'(\pi_y(a))} \\%
		U_1(\hxk)\arrow{r}{U^*_1}  & \hxk\arrow{r}{K} & \hyk \arrow{r}{U_2} & U_2(\hyk)
		\end{tikzcd}
		\]
		
		Hence $ \iota'\circ\pi_x \sim \iota'\circ\pi_y$. Now, suppose that $ \iota'\circ\pi_x \sim \iota'\circ\pi_y$. Then, there is a unitary $K: U_1(\hxk) \rightarrow U_2(\hxk)$ such that $K(\iota'\circ\pi_x)(a) = (\iota'\circ\pi_y)(a)K$, for any $a \in \mathcal{O}_{k(n-1)+1}$.
        Using Theorem \ref{thm:Upi=pi'U} we get $\pi_y(a)(U^*_2KU_1) = (U^*_2KU_1) \pi_x(a)$.
		Hence $\pi_x \sim \pi_y$.
		
		(3)\ Suppose that $\pi_x$ is not irreducible in $\hxk$.  Then, there is a $T \not \in \mathbb{C}1$ such that $\pi_x(a)T = T\pi_x(a)$, for any $a \in \mathcal{O}_{k(n-1)+1}$. By the previous diagram, this means that $\iota'(\pi_x(a))(UTU^*) = (UTU^*) \iota'(\pi_x(a))$. Suppose that $(UTU^*) = t$, where $t \in \mathbb{C}1$. Then, $T = U^*tU = t(U^*U) = t$, since $U$ is unitary. Therefore  $(UTU^*) \not \in \mathbb{C}1$, and so $\iota'(\pi_x)$ is not irreducible in $U(\hxk)$. An analogous argument implies that $\iota'(\pi_x)$ is not irreducible in $U(\hxk)$ if $\pi_x$ is not irreducible in $\hxk$. The result follows.
	\end{proof}

	
Theorem 6 in \cite{ar2008} implies that $\pi_x \sim \pi_y$ if and only if $x \sim y$, and that $\pi_x$ is irreducible in $\hxk$. Note that $\iota'\circ\pi_x: \mathcal{O}_{k(n-1)+1} \rightarrow B(U(\hxk))$ for every $x$. Therefore, using Theorem \ref{thm:Upi=pi'U}, we immediately obtain the following additional corollary.

\begin{corollary}
	\label{crl:pixy<->xy}
	Let $x,y \in [0,1[$. We have that $\iota'\circ\pi_x \sim \iota'\circ\pi_y$ if and only if $x \sim y$. Also, $\iota'\circ\pi_x$ is irreducible in $U(\hxk)$.
\end{corollary}

Given $g \in V_{k(n-1)+1}$, we have by Definition \ref{def:rho(g)} that $\rhoxk(g) = \pi_x(\Psi(g))$, where $\Psi(g) \in \mathcal{O}_{k(n-1)+1}$. Given $g \in V_{k(n-1)+1}$, Theorem \ref{thm:r'E=s'r} tells us then that $ \rhoxn(\E(g)) = \iota'(\rhoxk) = \iota'(\pi_x(\Psi(g)))$. Furthermore, by considering the restriction of $\pi_x$ to $\Psi(V_{k(n-1)+1})$ (a subset of $\mathcal{O}_{k(n-1)+1}$), Theorem \ref{thm:Upi=pi'U} gives us this corollary.
	
	\begin{corollary}
		\label{crl:Urho=rho'U}
		Let $x \in [0,1[$. Then:
\begin{enumerate}
  \item $\rhoxk$ and $\rhoxn\circ \E$ are unitarily equivalent of $V_{k(n-1)+1}$ on $\hxk$ and $U(\hxk)$, respectively;
  \item $\rhoxk \sim \rhoyk$ if and only if $ \rhoxn\circ \E \sim \rhoyn\circ \E$. Also, $\rhoxk$ is an irreducible representation in $\hxk$ if and only if $\rhoxn\circ \E$ is irreducible in $U(\hxk)$.
\end{enumerate}
	\end{corollary}

\subsection{Unitarily equivalence and irreducibility of $\{\rho_x\}_{x\in [0,1[}$}
\label{sec: unitequiv}

We now study the unitarily equivalence and irreducibility of the Higman-Thompson groups representations $\{\rho_x\}_{x\in [0,1[}$.
	
	

	We need to introduce some concepts and an auxiliary result. A probability measure $\mu$ on a set $Y$ is said to be finitely additive, if for any collection of finite pairwise disjoint subsets of $Y$, we have
	$$\mu(\bigcup_{i=1}^{m}A_i) = \sum_{i=1}^{m}\mu(A_i).$$
	A group action of a discrete group $G$ is said to be non-amenable, if there is no finitely additive probability measure $\mu$ in $Y$ such that, for any subset $A$ of $Y$, and some $g\in G$, we have
	\[\mu(gA) = \mu(A).\]
	Given a group $G$ with a subgroup $H$, if the action of $H$ in $Y$ is non-amenable, then so is the action of $G$ in $Y$ since a $G$-invariant probability measure $\mu$ would be $H$-invariant. We now quote a result from \cite{olesen}.
	
	\begin{theorem}\hspace{1pt}
		\label{thm:olesen}
		\begin{enumerate}
			\item The action $g\cdot y = g(y)$ of $V_2$ on $[0,1[$ is non-amenable.
			\item Suppose that $G$ is a discrete group acting on a set $X$, and let $\rho$ denote the induced representation on $l^2(X)$. Then the action of $G$ on $X$ is non-amenable if and only if there exist
$g_1, . . ., g_m$ in $G$ so that
			\[ \frac{1}{m} \left\| \sum_{k=1}^{m} \rho(g_k) \right\| < 1. \]
		\end{enumerate}	
	\end{theorem}
	
Let $\pi_x: \mathcal{O}_n\to B(\hxn)$ be the representation of $V_n$ as in Notation \ref{not1}.		
The following result is the last ingredient needed for the proof of Theorem \ref{thm:rho<->xy}.
	\begin{theorem}
		\label{thm:Crho=piO}
		Let $n \geq 2$. Then
		\[ C^*_{\rhoxn}(V_n) = \pi_x(\mathcal{O}_n). \]
	\end{theorem}

	\begin{proof}
	Since $\rhoxn(g) = \pi_x(\Psi(g))$, we have that $\rhoxn(V_n) \subset \pi_x(\mathcal{O}_n)$. Therefore, $C^*_{\rhoxn}(V_n) \subset \pi_x(\mathcal{O}_n)$. In order to prove that $\pi_x(\mathcal{O}_n) \subset C^*_{\rhoxn}(V_n)$, we will show that the set $\{S_1,S_2, \ldots, S_n\} $ is in $C^*_{\rhoxn}(V_n)$.
	
	We will start by proving that $\{S_1S_1^*,S_2S^*_2,\ldots,S_nS_n^*\} \subset C^*_{\rhoxn}(V_n) $. Let $J_i$ denote the set of $h \in V$ such that $h(z) = z$ for all $z \in [\frac{i-1}{n},\frac{i}{n}[$. If $a$, $b$ are maps such that $a(y) = b(y) = y$ for all  $ y \in [\frac{i-1}{n},\frac{i}{n}[$, then so does their composition, and their inverse. Hence, $J_i$ is a subgroup of $V_n$. We want to prove that for each $i$, there is a subgroup of $J_i$ isomorphic to $V_n$. Consider the subgroups $K_1$ and $K_2$ of $V_n$ defined by\\
	$K_1 = \left\{g\in V_n: g(x) = x \text{ for all } x\in [0,\frac{n-1}{n}[\right\},$ and\\
	$K_2 = \left\{g\in V_n: g(x) = x \text{ for all } x\in [\frac{1}{n},1[\right\}.$

	Notice that $K_1$ is a subgroup of $J_i$ for $1 \leq i \leq n-1$, and that $K_2$ is a subgroup of $J_n$. We now try to prove that $K_1$ and $K_2$ are isomorphic to $V_n$. Consider the map $$h: [0,1[\rightarrow[0,\frac{1}{n}[,\quad h(x) = x/n$$
	and the application $v:K_2 \rightarrow V_n$ such that $v(g) = hgh^{-1}.$
	Then, $v$ is a homomorphism of groups. Furthermore, notice that $v(g)$ is the identity map, if and only if $g$ is the identity of $K_2$. Therefore, $v$ is injective. Finally, given $g \in V_n$, the map $p \in K_2$ such that
	\[
	p(x) = \left\{
	\begin{array}{ll}
	(h^{-1}gh)(x) & x \in [0,\frac{1}{n}[  \\
	x & x \in [\frac{1}{n},1[
	\end{array}
	\right.
	\]
	satisfies $v(p) = g$. We conclude that $v$ is an isomorphism, and that $K_2$ is isomorphic to $V_n$. One can prove that $K_1$ is isomorphic to $V_n$ by considering the map
	$$h: [0,1[\rightarrow[\frac{n-1}{n},1[,\quad h(x) = (n-1+x)/n$$
	instead. Therefore, $V_n$ is embedded in $J_i$ for every $i$. Since by \cite{birget}, $V_2$ is embedded in $V_n$, we conclude that $V_2$ is embedded in $J_i$.
	
	We have that the elements of $J_i$ preserve the elements of the set $\mathcal{X}_i$, where $\mathcal{X}_i$ is the set
	\[ \mathcal{X}_1 = \left[0,\frac{1}{n}\right[\hspace*{4pt} \cap \hspace*{4pt}  \text{orb}_n(x),  \hspace*{30pt} \mathcal{X}_i = \left[\frac{i-1}{n},\frac{i}{n}\right[ \hspace*{4pt} \cap \hspace*{4pt} \text{orb}_n(x).  \]
	
	 The action $g\cdot y = g(y)$ of $V_2$ on $[0,1[$ is non-amenable (Theorem \ref{thm:olesen}). Since $V_2$ is embedded in $J_i$, the action $g\cdot y = g(y)$ of $J_i$ on $\mathcal{X}_i$ is also non-amenable. By Theorems \ref{thm:bigs4} and \ref{thm:olesen}, there exist $g_1, . . . , g_m \in J_i$ such that
	\[ \frac{1}{m} \| \sum_{k=1}^{m} \rhoxn(g_k) \| < 1, \]
	where we recall that by Theorem \ref{thm:bigs4} the action coincides with the representation $\rhoxn$.		
	Let $t = \frac{1}{m} \sum_{k=1}^{m} \rhoxn(g_k)$. Observe that $1 = S_1S_1^* + S_2S_2^* + \ldots +S_nS_n^*$, and that, given $a,b \in \{S_1S_1^*,S_2S_2^*,\ldots,S_nS_n^*\}$, we have $ab = \delta_{ab}a$ and $ta=at$, so that for any $k \in \mathbb{N}$
	\[ t^k = (t1)^k = (t(S_1S_1^* + \ldots + S_2S_2^*))^k = (tS_1S_1^*)^k + \ldots (tS_iS^*_i)^k + \ldots + (tS_nS_n^*)^k.\]
	On the other hand, since for any $g \in J_i$, we have $\rhoxn(g) = \ldots + S_iS_i^* +\ldots $, we conclude that $(tS_iS_i^*) = S_iS_i^*$, and thus $(tS_iS_i^*)^k = S_iS_i^*$. Since $\|t\| <1$ and $\|S_jS_j^*\| = 1$, we conclude that $t^k$ converges to $S_iS_i^*$. Thus, $S_iS_i^* \in C^*_{\rhoxn}(V_n)$ for all i. We will now use this to prove that $S_i \in C^*_{\rhoxn}(V_n)$ for all i.

	First notice that if $S_n$ is in $C^*_{\rhoxn}(V_n) $, then so is $S_i$ for any $i$. In fact, consider the following table
	\setcounter{MaxMatrixCols}{20}
	\[
	g=\begin{bmatrix}
	1 & \ldots & i1 & i2 & \ldots & in  & \ldots & a & \ldots & n1 & n2 & \ldots & nn\\
	1 & \ldots & n1 & n2 & \ldots & nn  & \ldots & a & \ldots & i1 & i2 & \ldots & in
	\end{bmatrix} \in V_n \]
	
	Then, noticing that for any $j \neq n$, $S_j^*S_n = 0$, and that only the words starting with $i$ are mapped to words starting with $n$, we obtain
	$\rhoxn(g)S_n =$
$$ (S_iS_1S^*_1S^*_n + S_iS_2S^*_2S^*_n + \ldots  + S_iS_nS^*_nS^*_n)S_n = S_i(S_1S^*_1 + S_2S^*_2 + \ldots + S_nS^*_n) = S_i$$
	Therefore $\rhoxn(g)S_n = S_i$. Thus, all we need to do now is to prove that $S_n \in C^*_{\rhoxn}(V_n) $. We start by proving that $S_iS_iS_i^* \in C^*_{\rhoxn}(V_n) $ for all $i$. To do this, consider the table on $v_n$ defined by $k=$
	\footnotesize{\[
	\begin{bmatrix}
	1 & 2 & \ldots & (i-1) & i1 & i2 & \ldots & i(i-1)  & ii & i(i+1) & \ldots & in & (i+1) & \ldots & n\\
	11 & 2 &\ldots & (i-1) & 12 & 13 & \ldots & 1i & i & 1(i+1) &\ldots & 1n & (i+1) & \ldots & n
	\end{bmatrix}.\]
}
	Then, $\rho(k)(S_iS_i^*) = (S_iS_iS_i^*)(S_iS_i^*) = S_iS_iS_i^*$. Therefore, $S_iS_iS_i^* \in C^*_{\rhoxn}(V_n) $.
	Let us consider the following table
	\[
	l=\begin{bmatrix}
	1  & \ldots & i & \ldots & n\\
	1 &\ldots & n & \ldots & i
	\end{bmatrix} \in V_n \]
	We have that $\rho(l)(S_iS_iS_i^*) = (S_nS_i^*)(S_iS_iS_i^*) = S_nS_iS_i^*$. Hence $S_nS_iS_i^* \in C^*_{\rhoxn}(V_n) $ for all $i$. Therefore
	\[ S_nS_1S_1^* + S_nS_2S_2^* + \ldots + S_nS_nS_n^* = S_n(S_1S_1^* + S_2S_2^* + \ldots + S_nS_n^*) = S_n(1)= S_n  \]
which is in $\in C^*_{\rhoxn}(V_n)$.
	This concludes the proof.
\end{proof}
	
	We will now use Corollary \ref{crl:pixy<->xy} and Theorem \ref{thm:Crho=piO} to prove  Theorem \ref{thm:rho<->xy} (1).

\begin{theorem}
\label{thm:rho<->xy}
		Let $n \geq 2$. Then:
\begin{enumerate}
  \item $\rhoxn \sim \rhoyn$ if and only if $x \sim y$;
  \item  $\rhoxn\circ \E \sim \rhoyn\circ \E$ if and only if $x \sim y$;
  \item $\rhoxn$ is an irreducible representations of $V_n$ on $\hxn$.
\end{enumerate}
	\end{theorem}
\begin{proof}
  (1)\ Suppose $x \sim y$. Then, $\hxn = \hyn$, and thus, $\rhoxn = \rhoyn$.
		Now, suppose that $\rhoxn \sim \rhoyn$. By definition, there must exist a unitary operator $K:\hxn \rightarrow \hyn$ such that, given $g \in V_n$,
		$\rhoxn(g) = K\rhoyn(g)K^*$
		which, can be rewritten as $\pi_x(\Psi(g)) = K\pi_y(\Psi(g))K^*$. Our goal is to show that, given $a \in \mathcal{O}_n$, $\pi_x(a) = K\pi_y(a)K^*$. In order to do this, let us consider the following subset of $\mathcal{O}_n$
		$B = \text{span}(\{\Psi(g): g \in V_n\})$
		Then, given $b \in B$, we have that
		$b = \sum_{i=1}^{m}c_i\Psi(g)$
		for some $c_i \in \mathbb{C}, m \in \mathbb{N}$. Furthermore, we have that
		\[ K\pi_y(b)K^* = K\pi_y(\sum_{i=1}^{m}c_i\Psi(g))K^* = \sum_{i=1}^{m}c_iK(\pi_y(\Psi(g)))K^* = \sum_{i=1}^{m}c_i\pi_x(\Psi(g)) = \pi_x(b).\]
		Let $a \in \mathcal{O}_n$. By Theorem \ref{thm:Crho=piO} there is a sequence $a_m$ in $B$ that converges to $a$. By continuity of $\pi_x$ and $\pi_y$, we conclude that $\pi_x(a) = K\pi_y(a)K^*$ and thus, $\pi_x \sim \pi_y$, which implies that $x \sim y$ by Corollary \ref{crl:pixy<->xy}.

  (2)\ It follows from Corollary \ref{crl:Urho=rho'U} and part (1) of this theorem.

  (3)\ Theorem \ref{thm:Crho=piO} also gives us that $\rhoxn$ is irreducible in $\hxn$, given that
	\[ (\rhoxn(V_n))^\prime = {(\text{span}(\rhoxn(V_n)))\ }^\prime = \overline{\text{span}(\{\rhoxn(g):g \in V_n\})\ }^\prime = C^*_{\rhoxn}(V_n)^\prime = \pi_x(\mathcal{O}_n)^\prime \]
	and then, since Theorem 6 in \cite{ar2008} implies that $\pi_x$ is irreducible in $\hxn$, $\rhoxn$ is irreducible in $\hxn$.
\end{proof}

\section*{Acknowledgements}
The first author would like to thank {\it Funda\c c\~{a}o Calouste Gulbenkian} Foundation. The second author was partially supported by FCT/Portugal through CAMGSD, IST-ID, projects UIDB/04459/2020 and UIDP/04459/2020.


\end{document}